\documentclass[11pt]{article}
\usepackage{amsmath,amsfonts, amssymb,amsthm,multicol,graphicx,tikz, enumitem}
\usepackage{makeidx}

\usepackage[utf8]{inputenc}
\usepackage{amsfonts, amssymb, bbm}
\usepackage{hyperref}
\newcommand{\cF}{\mathcal{F}}

\newcommand{\empt}{\varnothing}

\newtheorem{prop}{Proposition}   
\newtheorem{lem}[prop]{Lemma} 
\newtheorem{thm}[prop]{Theorem} 
\newtheorem{corollary}[prop]{Corollary}

\newtheorem{rem}[prop]{Remark}

\title{The $k$-color poset Ramsey number of the diamond is $2k$}
\author{Tom Bohman \, \and Will Cong \and Quentin Dubroff 
\and Teresa Huang \and Zheyuan Li \and BangTam Ngo}

\begin{document}

\maketitle

\begin{abstract}
The $k$-color Ramsey number of a poset $P$ is
the smallest dimension $d$ such that any $k$-coloring
of the $d$-dimensional Boolean lattice contains a monochromatic
induced copy of $P$. For $ j\ge 2$ the $j$-diamond is the poset that consists of a maximum 
element, a minimum element and an antichain of size $j$. We 
prove that the $k$-color Ramsey number of the $j$-diamond is at
most $jk$. Colorings of the Boolean lattice that assign sets of the
same size the same color provide matching lower bounds for $j \in \{2,3,4\}$.
\end{abstract}

\section{Introduction}

We consider Ramsey numbers for induced posets, a
line of inquiry that was initiated by Axenovich and Walzer \cite{AW}.
Let
$B_n$ denote the $n$-dimensional Boolean lattice, which is 
the subset lattice of an $n$-element set.
We define the $k$-color Ramsey number of a finite poset
$P$ to be the smallest $n$ such that any coloring of
the elements of $B_n$ with $k$ colors contains a 
monochromatic induced copy of $P$, where we say that 
$ {\mathcal P} \subseteq B_n$ is an induced copy of $P$ if there
is a bijection $ \varphi: P \to \mathcal P $ such that
$ x \le y $ if and only if $ \varphi(x) \subseteq \varphi(y)$.

As we are working within the Boolean lattice, which is
universal in the sense that every poset can be embedded
in a Boolean lattice of sufficiently large dimension, the natural
starting point for this study is the induced poset Ramsey numbers of 
the Boolean lattices themselves. There are a number of works that
address the 2-color case (i.e. the diagonal poset Ramsey numbers). 
See \cite{AW}, 
\cite{LT},
\cite{CoxStolee},  \cite{BohmanPeng} and \cite{Winterlower}. 
The currently best known bounds are
\[ 2.7n \le  R_2( B_n) \le n^2 - (1-o(1))n \log n, \] 
which were established by Ivan and Wessels \cite{IvanWessels}
and Axenovich and Winter \cite{AWinterDiag}, respectively. 

Here we focus on the multi-color poset Ramsey numbers. Let $ X = [n]$.
The Lubell mass of a collection of sets $ \cF \subseteq 2^{X}$ is defined to
be
\[  \ell( \cF) = \sum_{A \in \cF} \frac{1}{ |A| }.\]
M\'eroueh \cite{LYM} proved that for any fixed poset $P$ there is some constant $ C_P$ such that the
Lubell mass of 
$ \cF \subseteq 2^X$  that contains no copy of $P$ is at most $ C_P$. This result
directly implies that for any fixed poset $P$ we have $ R_k(P) \le C_P k$. 

Our focus is on the 2-dimensional Boolean lattice $B_2$, which is also
known as the diamond. This Ramsey number has been discussed in
the literature \cite{Vshaped} \cite{katona}. For example, Chen, Chen, Cheng, Li and Liu
note that forbidden poset theory implies $ R_k(B_2) \le 8k/3$ and ask
for the limit as $k$ tends to infinity of $ R_k(B_2)/k$. Here we resolve
this question in a strong form.
\begin{thm}\label{thm:diamond}
\[ R_{k}(B_2) = 2k.\]
\end{thm}
\noindent
This result is somewhat surprising in light of the recent progress on the
corresponding poset Tur\'an problem. It turns out that the maximum number
of sets in a family $ \cF \subset 2^X$ that does not contain a copy of $ B_2$
is significantly larger than the middle two layers of the Boolean lattice \cite{Tompkins}. In light of
this result, one might be tempted to believe that a number of such collections could be worked
into a partition of $2^X$ and thereby establish that the $k$-color Ramsey number is more
than $2k$. We show that this is not the case.

The proof of  Theorem~\ref{thm:diamond} adapts to the setting of 
general poset diamonds. For $j \ge 2$ define the $j-diamond$ to be the height 3 poset that
consists of a maximum element, a minimum element, and an antichain with
$j$ elements. Note that $ D_2 = B_2$. Our main result is the following.
\begin{thm}\label{thm:main}
Let $ j \ge 2$. Then
\[R_k( D_j) \le jk. \]
\end{thm}
\noindent
Note that $ B_{j-1}$ does not contain a copy of $ D_j$ for $ j=2,3,4$. It follows that in these 
cases the coloring of $B_n$ that assigns $j$ consecutive level-sets the same color does not
contain a monochromatic induced copy of $ D_j$. This then implies $ R_k(D_j) \ge jk $ for $ j \in \{2,3,4\}$.
\begin{corollary}
We have $ R_k( D_3) = 3k $ and $ R_k(D_4) = 4k$.
\end{corollary}
\noindent
The smallest diamond for which the $k$-color Ramsey number remains open is $ D_5$. The best known
bounds in this case are $ 4k \le R_k( D_5) \le 5k$.

\section{Proof of Theorem~\ref{thm:main}}

Our proof is inductive, and our induction
establishes some additional poset Ramsey bounds
along the way. In order to set up this induction
we introduce another poset. For $ j\ge 2$ define $ F_j$ to be the
poset with $j+1$ elements which consists of a maximum
element and an antichain on the remaining $j$ elements. Our proof bounds
multicolor Ramsey numbers where some colors are assigned to $F_j$ and
others are assigned to $ D_j$. 

Note that if we are interested in $ R_k( P_1, P_2, \dots, P_k)$
and $ F_j$ is one of the posets in the collection
$ P_1, \dots, P_k$ then the element $\empt$ of $ B_n$ does not play a meaningful 
role when we are
establishing the desired Ramsey property; indeed, if color
$c$ is assigned to $F_j$ then
$ \empt$ could be colored with $c$ and then $ \empt$ does
not play a role in a copy of $ F_j$ colored with color $c$. 
In light of this situation we define a variant of the $k$-color
Ramsey number. If $ P_1, P_2, \dots, P_k$ are posets we define
$ R_k^*( P_1, \dots, P_k) $ to be the smallest $n$ such that
for any coloring of $B_n \setminus \{ \empt\}$ with colors 
$ \{c_1, \dots, c_k \}$ there is an index $i$ such that there
is a monochromatic copy of $P_i$ in color $c_i$. Our observation
here is 
\begin{equation} \label{eq:ignoringempty}
F_j \in \{ P_1, \dots, P_k \} \ \ \Rightarrow \ \ R_k( P_1, \dots P_k) = R_k^*( P_1, \dots, P_k). 
\end{equation}

\begin{lem} \label{lem:fans}
\[ R_k( F_j) = R_k^*( F_j) \le k(j-1) +1 \]
\end{lem}
\begin{proof}
Let $j$ be fixed.
We go by induction on $k$. We clearly have $ R^*_1( F_j) \le  j$.

Now assume $k \ge 2$ and $ R^*_{k-1}(F_j) \le (k-1)(j-1)+1$ and consider a $k$-coloring of
$ 2^{ X} \setminus \{ \empt\}$ where $ |X| = \ell := k(j-1)+1$. 
Let $c$ be the color used on the set $X$. Let $ {\mathcal G}_c$ be the collection
of subsets of $X$ that are also colored with $c$. If there is no copy
of $F_j$ in ${\mathcal G}_c$ then the maximum size of an antichain in $ {\mathcal G}_c$
is $j-1$. It then follows from Dilworth's Theorem that $ {\mathcal G}_c$ can be
partitioned into at most $j-1$ chains. For each of these chains there is some particular
element of $X$ that appears in every set in the chain. Thus, there exists a set $S \subset X$ such
that $|S| \le j-1$ such that every set in $ {\mathcal G}_c $ intersects $S$.

Restricting our attention to $ 2^{X \setminus S}\setminus \{ \empt \}$ 
we have a coloring with at most $ k -1$ colors. As
\[ |X \setminus S| \ge \ell - (j-1) = (k-1)(j-1) +1, \]
the inductive assumption implies that this coloring contains
a monochromatic copy of $ F_j$.
\end{proof}
\begin{rem} The bound in Lemma~\ref{lem:fans} may be far from the truth for $j$ large relative to $k$. 
For example, we can 
interpret $R_1(F_j)$ to be the small dimension $n$ such that $B_n$ contains an antichain 
with $j$ elements, which is logarithmic in $j$ for large $j$, by Sperner.
\end{rem}
\begin{rem} The $j=2$ case of Lemma~\ref{lem:fans} was established by Chen et al. \cite{Vshaped}.
\end{rem}
Let $ j \ge 2$ be fixed.
For non-negative integers $a,b$, not both zero, define 
\[ f(a,b) = R_{a+b}( \underbrace{D_j, \dots D_j}_{ a \text{ times}}, \underbrace{F_j, \dots, F_j}_{b \text{ times}}). \]
In words, $ f(a,b)$ is the
$(a+b)$-color poset Ramsey number in which $a$ colors are assigned to the poset $D_j$ and $b$ colors
are assigned to $ F_j$. We prove the following generalization of Theorem~\ref{thm:main}.
\begin{thm} If $ j \ge 2$ then
\begin{equation} 
\label{eq:full}
f(a,b) \le \begin{cases} ja + (j-1)b + 1 & \text{ if } b>0 \\
ja & \text{ if } b =0. 
\end{cases} 
\end{equation}
\end{thm}

\begin{proof}
We go by induction on the number of colors, which is $ k = a +b$. And for
each fixed number of colors we go by induction on $a$.

We clearly have $ R_1(D_j), R_1( F_j) \le j$, which establishes the base case for the
induction on the number of colors.

For a fixed number of colors $k\ge 2$ we note that Lemma~\ref{lem:fans} establish the
base case, which is the statement $ f(0,k) = R_k( F_j) = R^*_k( F_j) \le (j-1)k+1$.

So it suffices to establish the bound on $ f(a,b)$ given in (\ref{eq:full}) for $a \ge 1, a+b \ge 2$
while assuming the bounds on $ f(a-1, b+1)$ and $ f(a-1,b)$, as well as the bound on $ f(a,b-1)$
if $b>0$, given in (\ref{eq:full}). We emphasize that we can appeal to (\ref{eq:ignoringempty}), and thereby ignore the color assigned to $ \empt$, 
whenever we are in a situation when at least one color is assigned to $F_j$.

We divide the argument into three cases. First suppose $ b > 1$. Recalling (\ref{eq:ignoringempty}), we can
ignore the color assigned to $ \empt$. Consider
a coloring $\chi$ of $ 2^X \setminus \{ \empt\} $, where 
$|X| =  ja + (j-1)b + 1$,  with $a+b$ colors where $C_1$ is the set of
$a$ colors assigned to $ D_j$ and $C_2$ is the set of $b$ colors 
assigned to colors $F_j$. We consider two subcases.
\begin{itemize}
\item[Case 1.] There is an element $ x \in X$ such that $ \chi(\{x\}) \in C_1$. Here we set
$ Y = X \setminus \{x\} $ and consider the coloring $\psi$ of $ 2^Y \setminus \{ \empt \}$ given
by $ \psi(A) = \chi( A \cup \{x\}) $.  As
\[ |Y| = |X| -1 = j(a-1) + (j-1)(b+1) +1,\]
the inductively assumed bound on $ f(a-1, b+1)$ together with (\ref{eq:ignoringempty}) implies that
this coloring contains a copy of $ D_j$ colored with some color in $C_1 \setminus \{ \chi(\{x\}) \}$, 
a copy of $F_j$ colored with some color $ C_2$ or a copy of $ F_j$ in color $ \chi( \{x\})$. In
the first two cases we immediately have the desired structure in the coloring $ \chi $. In the
last we note that the copy of $ F_j$ in color $ \chi( \{x\}) $ together with $\{x\}$ is a copy of
$D_j$ in color $\chi( \{x\}) $.

\item[Case 2.] For all  $ x \in X$ we have $ \chi(\{x\}) \in C_2$. 
By pigeonhole there is a color $ c \in C_2 $ and a set $ W \subset X$ such that $ |W| = j -1$
and $ \chi( \{w\}) = c$ for all $ w \in W$. (N.b. Pigeonhole actually gives more than this, but
this is what we need for the proof.)

Now we define a coloring on $ 2^{ X \setminus W} \setminus \{ \empt \}$. We begin by letting $ {\mathcal H}_c $
be all sets $  \empt \neq A \subseteq X \setminus W$ such that there exists $ \empt \neq B \subseteq A$ such
that $ \chi( B) = c$. In words $ {\mathcal H}_c$ is the up-set generated by all nonempty subsets of $ X \setminus W$ that
are colored with color $ c$. We define a coloring $ \psi $ of $ 2^{ X \setminus W} \setminus \{ \empt \}$
as follows:
\[ \psi(A) = \begin{cases} \chi(A) & \text{ if } A \not\in {\mathcal H}_c \\
\chi(A \cup W) & \text{ if } A \in {\mathcal H}_c. 
\end{cases}
\]

Suppose there is some set $A$ such that $ \psi(A) = c$. Then $ A \in {\mathcal H}_c $,
$ \chi(A \cup W) = c$, and there is some $B \subseteq A$ such that $ \chi(B) = c$. In this situation the
sets $ \{w\}$ for each $w \in W$, $B$ and $A\cup W$ form a copy of $F_j$ in color $c$ in the coloring $ \chi$. On
the other hand, if 
there is no set $A \in 2^{ X \setminus W} \setminus \{ \empt \}$ such that $ \psi(A) = c$ then the coloring $ \psi$ uses
$ a+b-1$ colors. As
\[ |X \setminus W| = ja + (j-1)b +1 -  (j-1) = ja + (j-1)(b-1) +1,\]
the inductive assumption on $ f(a,b-1)$ together with (\ref{eq:ignoringempty}), and recalling that we are assuming $ b-1 \ge 1$, implies that $ \psi$ has a copy of $ D_j$ in some
color in $C_1$ or a copy of $F_j$ in some color in $ C_1 \setminus \{c\}$. Observe that the poset induced on 
\[ \left(  2^{X \setminus W} \setminus
\mathcal{H}_c \right) \bigcup \left\{ A \cup W: A  \in \mathcal{H}_c  \right\}\] is order isomorphic to
the Boolean cube on $|X \setminus W|$ elements, so a monochromatic induced
copy of a poset in $\psi$ is a monochromatic copy in $\chi$.
%
%

\end{itemize}

Now consider the case $b=1$. Consider
a coloring $\chi$ of $ 2^X \setminus \{ \empt\} $, where 
$|X| =  ja + (j-1) + 1$,  with $a+1$ colors where $C_1$ is the set of
$a$ colors assigned to $ D_j$ and $c$ is the single color assigned to $F_j$. Now
we note that if there $x \in X$ such that $ \chi( \{x\}) \neq c$ then we 
apply the argument in Case 1 above to conclude that the desired monochromatic
induced poset exists in the coloring $ \chi$.  Suppose, on the other hand, that $ \chi( \{x\}) = c$ for
all $ x \in X$.  Then we have a copy of $ F_j$ in color $c$ if any set on at least $j$ elements is
colored with color $c$. Assume that this is not the case, and let $Y \in \binom{X}{j}$ be arbitrary. Consider the coloring $ \psi$ of
$ 2^{X \setminus Y}$ given by $ \psi(A) = \chi( A \cup Y)$. We emphasize that in this
case we include $\empt$ in the coloring. As we have $ | X \setminus Y| =ja $ our
inductive assumption on $ f(a,b-1) = f(a,0)$ implies that there is a mononchromatic copy of
$D_j$ in the coloring $ \psi$. This is also a monochromatic copy of $D_j$ is $ \chi$.

Finally, consider $b=0$.
Let $ X$ be a set such that $ |X| = ja $ and consider a coloring $ \chi$ of $2^X$ with
$a$ colors. Let $c = \chi(\empt)$. By our inductive assumption
$ f(a-1,1) \le (a-1)j + (j-1) + 1 = aj$. Recalling (\ref{eq:ignoringempty}), it follows that the coloring of $ 2^X \setminus \{ \empt \}$ given by the restriction of $ \chi$ has
a monochromatic $ D_j$ in a color other than $c$ or a copy of $ F_j$ in color $c$. In either case, $\chi$ itself has monochromatic copy of $ D_j$.

\end{proof}

\medskip

\noindent
{\bf Acknowledgment.} This research was supported in part by the Natasha Parker Endowed Research Fund. The first author was supported by NSF grant DMS-2246907. 

\medskip

\noindent
{\bf AI statement.} AI tools were not used in the production
of this research or the composition of this manuscript.

\bibliographystyle{plain}
\bibliography{reference}

\end{document}